\documentclass[11pt,reqno]{amsart}

\usepackage{amsthm,amsmath,amssymb}
\usepackage{graphicx}
\usepackage{float}
\usepackage[colorlinks=true,
linkcolor=blue,
urlcolor=blue,
citecolor=blue]{hyperref}
\usepackage{mathtools}
\usepackage{relsize}
\usepackage{ytableau}
\usepackage{tikz}
 \usepackage{enumerate}
\usepackage[toc,page]{appendix}
\usepackage{lscape}
\usepackage{multirow}
\usepackage{adjustbox,expl3,etoolbox} 
\usepackage[margin = 3.5cm]{geometry}

\newtheorem{thm}{Theorem}[section]
\newtheorem{lem}[thm]{Lemma}
\theoremstyle{definition}

\newtheorem{prop}[thm]{Proposition}
\newtheorem{conj}[thm]{Conjecture}
\newtheorem{cor}[thm]{Corollary}

\newtheorem{rmk}[thm]{Remark}
\newtheorem{prob}[thm]{Problem}

\newtheorem{qst}[thm]{Question}

\DeclareMathOperator\sym{Sym}
\DeclareMathOperator\alt{Alt}

\DeclareMathOperator\Der{Der}
\DeclareMathOperator{\agl}{AGL}
\DeclareMathOperator\gl{GL}

\DeclareMathOperator\Fix{F}
\DeclareMathOperator{\pg}{PG}
\DeclareMathOperator{\ag}{AG}
\DeclareMathOperator{\stab}{Stab}

\letcs\replicate{prg_replicate:nn}

\begin{document}	

\title{On complete multipartite derangement graphs}

\author[A. Sarobidy Razafimahatratra]{Andriaherimanana Sarobidy Razafimahatratra}
\address{Department of Mathematics and Statistics, University of Regina,
  Regina, Saskatchewan S4S 0A2, Canada}\email{sarobidy@phystech.edu}
	\thanks{Department of Mathematics and Statistics, University of Regina,
		Regina, Saskatchewan S4S 0A2, Canada}

\begin{abstract}

 Given a finite transitive permutation group $G\leq \sym(\Omega)$, with $|\Omega|\geq 2$, the derangement graph $\Gamma_G$ of $G$ is the Cayley graph $\operatorname{Cay}(G,\Der(G))$, where $\Der(G)$ is the set of all derangements of $G$. Meagher et al. [On triangles in derangement graphs, {\it J. Combin. Theory Ser. A}, 180:105390, 2021] recently proved that $\sym(2)$ acting on $\{1,2\}$ is the only transitive group whose derangement graph is bipartite and any transitive group of degree at least three has a triangle in its derangement graph. They also showed that there exist transitive groups whose derangement graphs are complete multipartite.
  
 This paper gives two new families of transitive groups with complete multipartite derangement graphs. In addition, we prove that if $p$ is an odd prime and $G$ is a transitive group of degree $2p$, then the independence number of $\Gamma_{G}$ is at most twice the size of a point-stabilizer of $G$.
\end{abstract}



\date{February 11, 2021}

\maketitle


\section{Introduction}

This paper is concerned with Erd\H{o}s-Ko-Rado (EKR) type theorems for finite transitive groups. The classical EKR Theorem is stated as follows.
\begin{thm}[Erd\H{o}s-Ko-Rado \cite{erdos1961intersection}]
	Suppose that $n,k\in \mathbb{N}$ such that $2k\leq n$. If $\mathcal{F}$ is a family of 
	$k$-subsets of $[n]:=\{1,2,\ldots,n\}$ such that $A\cap B \neq \varnothing$ for all $A,B \in \mathcal{F}$, then $|\mathcal{F}|\leq \binom{n-1}{k-1}$. Moreover, if $2k<n$, then equality holds if and only 
	if 
	$\mathcal{F}$ consists of all the subsets which contain a fixed element of $[n]$.\label{thm:EKR}
\end{thm}

The EKR theorem has been well studied and generalized for numerous combinatorial objects in the past 50 years \cite{cameron2003intersecting,Frankl1977maximum,deza1983erdos,frankl1986erdos,godsil2016erdos,meagher2011erdHos,meagher2016erdHos,spiga2019erdHos,wilson1984exact}. Of interest to us is the 
generalization of Theorem~\ref{thm:EKR} for the symmetric group by Deza and Frankl in 
\cite{Frankl1977maximum}.

Given a finite transitive permutation group $G\leq \sym(\Omega)$, we say that the permutations $\sigma,\pi \in G$ are \emph{intersecting} if $\omega^\sigma = \omega^\pi$, for some $\omega \in \Omega$. A subset or family $\mathcal{F}$ of $G$ is \emph{intersecting} if any two permutations of $\mathcal{F}$ are intersecting. 

\begin{thm}[Deza-Frankl,\cite{Frankl1977maximum}]
	Let $\Omega$ be a set of size $ n\geq 2$. If $\mathcal{F} \subset \sym(\Omega)$ is an intersecting family, then $|\mathcal{F}|\leq (n-1)!$.\label{thm:deza-frankl}
\end{thm}
The characterization of the maximum intersecting families of $\sym(\Omega)$ was solved almost three decades later by Cameron and Ku \cite{cameron2003intersecting}, and independently by Larose and Malvenuto \cite{larose2004stable}.
\begin{thm}[\cite{cameron2003intersecting,larose2004stable}]
	Let $\Omega$ be a set of size $ n\geq 2$. If $\mathcal{F}\subset \sym(\Omega)$ is an intersecting family of maximum size, that is $|\mathcal{F}| = (n-1)!$, then $\mathcal{F}$ is a coset of a stabilizer of a point of $\sym(\Omega)$. In particular, there exist $i,j\in \Omega$ such that 
	\begin{align*}
		\mathcal{F} = \left\{ \sigma \in \sym(\Omega) \mid i^\sigma =j \right\}.
	\end{align*}\label{thm:cameron-ku}
\end{thm}

The natural question that arises is whether analogues of Theorem~\ref{thm:deza-frankl} and Theorem~\ref{thm:cameron-ku}  hold for different subgroups of $\sym(\Omega)$, i.e., permutation groups of degree $n$. 
All groups considered in this paper are finite. 
We are interested in the following  extremal problem. 
\begin{prob}
	Let $G\leq \sym(\Omega)$ be transitive.
	\begin{itemize}
		\item[(1)] What is the largest size of  an intersecting family of $G$?
		\item[(2)] If $\mathcal{F}$ is an intersecting family of $G$ of maximum 
		size, then describe the structure of $\mathcal{F}$.
	\end{itemize}\label{main-prob}
\end{prob}

Not surprisingly, the answer to this problem depends on the structure of the subgroup of $\sym(\Omega)$. For instance, if $\sigma_1=(1\ 2)(3\ 4)$, $\sigma_2 = (3\ 4)(5\ 6)$ and $\tau = (1\ 3\ 5)(2\ 4\ 6)$ are permutations of $\Omega = \{1,2,3,4,5,6\}$, then $\langle \sigma_1,\sigma_2,\tau \rangle$ has its point-stabilizers of size $2$ but $$\displaystyle\mathcal{F} = \{id,(1\ 2)(3\ 4),(3\ 4)(5\ 6),(1\ 2)(5\ 6)\}$$ is a larger intersecting family. More examples of transitive permutation groups having larger intersecting families than point-stabilizers are given in \cite{meagher2015erdos,li2020erd,meagher180triangles}. Due to this, we consider the following definitions. We say that the group $G$ has the \emph{EKR property} if any intersecting family of $G$ has size at most $\frac{|G|}{|\Omega|}$ and $G$ has the \emph{strict-EKR property} if it has the EKR property and an intersecting family of size $\frac{|G|}{|\Omega|}$ is a coset of a stabilizer of a point.

A typical approach in solving EKR-type problems is reducing it into a problem on a graph theoretical 
invariant. The \emph{derangement graph} $\Gamma_{G}$ of $G\leq \sym(\Omega)$ is the graph whose vertex set is $G$ and two permutations $\sigma,\pi$ are adjacent if and only if they are not intersecting; that is, $\omega^\sigma \neq \omega^\pi$, for every $\omega\in \Omega$. In other words, $\Gamma_{G}$ is the Cayley graph $\operatorname{Cay}(G,\Der(G))$, where $\Der(G)$ is the set of all derangements of $G$. Then, a family $\mathcal{F}\subset G$ is intersecting if and only if $\mathcal{F}$ is an \emph{independent set} or a \emph{coclique} of the derangement graph $\Gamma_{G}$. Therefore, Problem~\ref{main-prob} is equivalent to finding the size of the maximum cocliques $\alpha(\Gamma_{G})$ and the structures of the cocliques of size $\alpha(\Gamma_{G})$.

Our long term objective is to classify the transitive permutation groups that have the EKR property and strict-EKR property. A big step toward this classification is the result of Meagher, Spiga and Tiep \cite{meagher2016erdHos}, which says that every finite $2$-transitive group has the EKR property. More examples of primitive groups having the EKR property are given in \cite{ahmadi2014new,ahmadi2015erdHos,behajaina20203,ellis2011intersecting,meagher20202,meagher2011erdHos,spiga2019erdHos}. 

We are motivated to find more transitive groups that do not have the EKR property. The group $\langle \sigma_1,\sigma_2,\tau \rangle$ given above is special in the sense that its derangement graph is a complete tri-partite graph. A recent result by Meagher, Spiga and the author \cite{meagher180triangles} brought to light the existence of many transitive groups that do not have the EKR property. The most important of these are the transitive groups whose derangement graphs are complete multipartite graphs. If $G\leq \sym(\Omega)$ is transitive and $\Gamma_G$ is a complete multipartite graph, then it is easy to see that the part $H$ of $\Gamma_{G}$, which contains the identity element $id$, consists of the elements with at least one fixed point. Moreover, every element of $G\setminus H$ is a derangement. Therefore, $H$ is a maximum coclique of $\Gamma_{G}$ and $H$ is the union of all the point-stabilizers of $G$. Thus, $G$ does not have the EKR property unless $H = \{id\}$. 
An important result  on the structure of derangement graphs of transitive groups is given in the next theorem.

\begin{thm}[\cite{meagher180triangles}]
	Let $G\leq \sym(\Omega)$ be transitive. Then, $\Gamma_G$ is bipartite if and only if $|\Omega|\leq 2$. Further, if $|\Omega|\geq 3$, then $\Gamma_G$ contains a triangle.\label{lem:not-bipartite}
\end{thm}

Our motivation for this work is to find more transitive groups having complete multipartite derangement graphs.
In this paper, we give two infinite families of transitive groups whose derangement graphs are complete multipartite. 
Our main results are stated as follows.

\begin{thm}
	Let $p$ be a prime and let $q=p^k$, for some $k\geq 1$. Then, there exists a transitive group $G_q$, of degree $q(q+1)$, such that $\Gamma_{G_q}$ is a complete $(q+1)$-partite graph.\label{thm:main}
\end{thm}

%

The following was conjectured in \cite{meagher180triangles} on the existence of complete multipartite derangement graphs.
\begin{conj}
	 If $n$ is even but not a power of $2$, then there is a transitive group $G$ of degree $n$ such that $\Gamma_G$ is a complete multipartite graph with $n/2$ parts.\label{conj:complete-multipartite}
\end{conj}

A transitive group of degree $n = 2\ell$, where $\ell$ is odd, with a complete $\ell$-partite derangement graph was given  in \cite[Lemma~5.3]{meagher180triangles}. We generalize this construction to find another family of transitive groups with complete multipartite derangement graphs. This result further reinforces Conjecture~\ref{conj:complete-multipartite}.
\begin{thm}
	For any odd $\ell$, there exists a transitive permutation group of degree $4\ell$ whose 
	derangement graph is a complete $2\ell$-partite graph.\label{thm:main-construction}
\end{thm}

The \emph{intersection density} $\rho(G)$ of a permutation group $G$ was introduced in 
\cite{li2020erd,meagher180triangles} as the ratio between the size of the largest intersecting 
families of $G$ and the size of the largest point-stabilizer of $G$. That is, if $G\leq \sym(\Omega)$, then 
\begin{align}
	\rho(G):= \frac{\max\{  |\mathcal{F}| \ :  \mathcal{F} \subset G  \mbox{ is 
		intersecting}\}}{\max_{\omega\in \Omega} |G_\omega|}.\label{eq:intersection-density}
\end{align}
For any $n\in \mathbb{N}$, we define $\mathcal{I}_n := \{ \rho(G)  \mid G \mbox{ is transitive of degree }n \}$ and $I(n) := \max 
\mathcal{I}_n$. The following was conjectured in \cite{meagher180triangles}.
\begin{conj}[\cite{meagher180triangles}]\hfill
	\begin{enumerate}
		\item If $n =pq$ where $p$ and $q$ are odd primes, then $I(n) = 1$.
		\item If $n = 2p$ where $p$ is prime, then $I(n) = 2$.\label{conj:4}
	\end{enumerate}\label{conj:main}
\end{conj}
In this paper, we also prove that Conjecture~\ref{conj:main} \eqref{conj:4} holds.
\begin{thm}
	If  $p$ is an odd prime, then $I(2p) = 2$.\label{thm:main2}
\end{thm}

This paper is organized as follows. In Section~\ref{sect:background}, we give some background results on complete multipartite derangement graphs and some properties of the intersection density of transitive groups. In Section~\ref{sect:proof1}, Section~\ref{sect:proof2}, and Section~\ref{sect:proof3}, we give the proof of Theorem~\ref{thm:main}, Theorem~\ref{thm:main-construction}, and Theorem~\ref{thm:main2}, respectively.


\section{Background}\label{sect:background}

Throughout this section, we let $G\leq \sym(\Omega)$ be a transitive group and  $|\Omega| = n$. 
\subsection{Bound on maximum cocliques}
 We recall that the problem of finding the size of the maximum intersecting families of $G$ is equivalent to finding the size of the maximum cocliques of 
$\Gamma_{G}$. We give a classical upper bound on the size of the largest cocliques in vertex-transitive 
graphs (i.e., graphs whose automorphism groups act transitively on their vertex sets). As the derangement graph of an arbitrary finite permutation group is a Cayley graph, the right-regular representation of $G$ acts regularly on $V(\Gamma_{G})$. In other words, $\Gamma_{G}$ is vertex transitive.

\begin{lem}[\cite{godsil2016erdos}]
	If $X = (V,E)$ is a vertex-transitive graph, then $\alpha(X) \leq \frac{|V(X)|}{\omega(X)}$. 
	Moreover, equality holds if and only if a maximum coclique of $X$ intersects each maximum clique at exactly one 
	vertex.\label{lem:clique-coclique}
\end{lem}
Lemma~\ref{lem:clique-coclique} can be used to prove the EKR 
property of groups. For instance, one can prove that $\sym(n)$, for $n\geq 3$, has the EKR property 
\cite{cameron2003intersecting,Frankl1977maximum,godsil2009new} by showing 
first that $\omega(\Gamma_{\sym(n)}) = n$ (a clique of $\Gamma_{\sym(n)}$ is induced by a Latin square of size $n$) and applying Lemma~\ref{lem:clique-coclique}. A subset $S \subset G$ 
with $|S| = n$ that forms a clique in $\Gamma_{G}$ is called a \emph{sharply} $1$-\emph{transitive set}. 
It is well-known that a transitive group need not have a sharply $1$-transitive set. Therefore, 
Lemma~\ref{lem:clique-coclique} does not hold with equality for the derangement graphs of many transitive groups.

\subsection{Intersection density}
By \eqref{eq:intersection-density}, the intersection density of the transitive group $G$ is the rational number
\begin{align*}
	\rho(G) : = \frac{\max |\{\mathcal{F}\subseteq G\mid \mathcal{F} \mbox{ is intersecting}\}|}{|G_\omega|},
\end{align*} 
where $\omega\in \Omega$. 

The major result in \cite{meagher180triangles} (see also Theorem~\ref{lem:not-bipartite}) asserts that the intersection density of the transitive 
group $G$ cannot be equal to $\frac{n}{2}$. This is equivalent to saying that the derangement graph of transitive 
groups cannot be bipartite if $n\geq 3$ (see \cite{meagher180triangles}). 
It is also proved in \cite{meagher180triangles} that for any transitive group $K$ of degree $n$, 
$\rho(K)$ is in the interval $\left[ 1,\frac{n}{3}\right]$. We note that $\rho(K) = 1$ if and only if $K$ 
has the EKR property. Moreover, the upper bound $\frac{n}{3}$ is sharp since there are transitive groups whose derangement graphs are complete tri-partite graphs 
\cite[Theorem~5.1]{meagher180triangles}. It is conjectured that the only transitive groups that attain 
the upper bound are those with complete tri-partite derangements graphs.

The study of the intersection density (see \cite{li2020erd,meagher180triangles}) of a transitive group was mainly motivated by studying how far from having the EKR property a transitive group can be. The intersection density, therefore, is a measure of the EKR property for transitive groups.

We make the following conjecture based on computer search using \verb*|Sagemath| \cite{sagemath}.
\begin{conj}
	For any $n\geq 3$, almost all elements of the set $\mathcal{I}_n$ are integers. That is, $$\frac{\left|\left\{ \rho(G) \mid \ G \mbox{ is transitive of degree $n$} \right\}\cap \mathbb{N}\right|}{|\mathcal{I}_n|} \xrightarrow[n\to \infty]{} 1.$$
\end{conj} 

Note that the intersection density of a transitive group can be non-integer. For example, the transitive groups of degree $n$ and number $k$ in the \verb|TransitiveGroup| function of \verb|Sagemath|, with $(n,k) \in \{(12,122),(12,93)\}$, have non-integer intersection density. \verb|TransitiveGroup|(12,122) and \verb|TransitiveGroup|(12,93) have intersection density  equal to $\frac{3}{2}$ and $\frac{17}{16}$, respectively.

\begin{prop}
	If the derangement $\Gamma_{G}$ has a clique of size $k$, then $\rho(G) \leq \frac{n}{k}$.
\end{prop}
\begin{proof}
	The proof follows by applying Lemma~\ref{lem:clique-coclique}.
\end{proof}

\subsection{Complete multipartite derangement graphs}
The transitive groups with complete multipartite derangement graphs are the most natural examples 
of groups that do not have the EKR property. In this subsection, we give some properties of transitive groups whose derangement graphs are complete multipartite.

The following lemma is a straightforward observation on the intersecting subgroups of $G$.
\begin{lem}[\cite{li2020erd,meagher180triangles}]
	Let $G\leq \sym(\Omega)$ and let $H\leq G$. Then, $H$ is intersecting if and only  if $H$ does not 
	have  any derangement.\label{lem:intersecting-subgroups}
\end{lem}
The next lemma illustrates that transitive groups with complete multipartite derangement graphs have 
a very distinct algebraic structure.
\begin{lem}[\cite{meagher180triangles}]
	If $G\leq \sym(\Omega)$ is transitive such that $\Gamma_{G}$ is a complete multipartite graph, then 
	$G$ is imprimitive.\label{lem:imprimitive}
\end{lem}

A transitive group whose derangement graph is a complete multipartite graph is uniquely determined by 
a particular subgroup of $G$. We define $\Fix(G)$ to be the subgroup of $G$ generated by all the permutations of $G$ with at 
least one fixed point. That is,
\begin{align*}
	\Fix(G) := \left\langle  \bigcup_{\omega\in \Omega } G_\omega  \right\rangle.
\end{align*}

\begin{prop}
	The subgroup $\Fix(G)$ is a normal subgroup of $G$.
\end{prop}

\begin{proof}
	The proof follows from the fact that $\Fix(G)$ is generated by all point-stabilizers.
\end{proof}
Note that Lemma~\ref{lem:imprimitive} follows from the normality of $\Fix(G)$ as its orbits form a non-trivial system of imprimitivity of $G$ acting on $\Omega$. 

A characterization of transitive groups with complete multipartite derangement graphs is given in the next lemma.

\begin{lem}[\cite{meagher180triangles}]
	Let $G\leq \sym(\Omega)$ be transitive. The graph $\Gamma_{G}$ is complete 
	multipartite if and only if $\Fix(G)$ is intersecting. Moreover, if $\Gamma_{G}$ is a complete multipartite 
	graph, then the number of parts of $\Gamma_{G}$ is $[G:\Fix(G)]$.\label{lem:complete-multipartite}
\end{lem}
Suppose that $\Gamma_{G}$ is a complete multipartite graph. When the subgroup $\Fix(G)$ is the trivial group $\{id\}$, then $\Gamma_{G}$ is the complete multipartite 
graph that has $|G|$ parts of size $1$. In other words, $\Gamma_{G}$ is the complete graph $K_{|G|}$. 
When $\Fix(G) = G$, then $\Fix(G)$ cannot be intersecting since by 
Lemma~\ref{lem:intersecting-subgroups}, this would contradict the celebrated 
theorem of Jordan \cite{jordan1872recherches,serre2003theorem} on the existence of derangements in finite transitive groups. Hence, we say that $\Gamma_{G}$ is 
a \emph{ non-trivial complete multipartite graph} if $1<|\Fix(G)|<|G|$. In this paper, we are only interested in transitive groups with non-trivial complete multiplartite derangement graphs.

Next, we study the nature of $\Fix(G)$. If $\Fix(G)$ is intersecting, then by Lemma~\ref{lem:intersecting-subgroups}, $\Fix(G)$ is derangement-free. Thus, 
\begin{align*}
	\Fix(G) = \bigcup_{\omega \in \Omega} G_\omega.
\end{align*}
Recall that if $K\leq \sym(\Omega)$ and $\omega \in \Omega$, then the orbit of $K$ containing $\omega$ is denoted by $\omega^K$. Moreover, if $S \subset \Omega$, then the setwise stabilizer of $S$ in $K$ is denoted by $K_{\{S\}}$.

The following lemma is a standard result in the theory of permutation groups.
\begin{lem}
	Let $G\leq \sym(\Omega)$ and $\omega\in \Omega$. If $H$ is a {non-trivial} subgroup of $G$ 
	containing $G_\omega$, then $G_{\{\omega^H\}} = 
	H$.\label{lem:setwise-stabilizer}
\end{lem}

\begin{cor}
	Let $G\leq \sym(\Omega)$ be transitive and let $K$ be the subgroup of $G$ fixing the system of imprimitivity $\left\{ 
	\omega^{\Fix(G)} \mid \omega\in \Omega \right\}$. Then $K = \Fix(G)$.
\end{cor}

\begin{proof}
	Since $\Fix(G)$ is generated by the point-stabilizers, by the previous lemma, we have
	\begin{align*}
	K &= \bigcap_{\omega \in \Omega} G_{\{ \omega^{\Fix(G)} \}} = 
	\bigcap_{\omega \in \Omega} \Fix(G) = \Fix(G).
	\end{align*}
\end{proof}

\begin{rmk}
	A representation of the derangement graph of the transitive group $G$ as a complete mutlipartite graph is unique. This is due to the fact that the part of $\Gamma_{G}$, which contains the identity element, must be equal to $F(G)$.
\end{rmk}

%

\section{Proof of Theorem~\ref{thm:main}}\label{sect:proof1}

In this section, we describe the action of $\agl(2,q)$ on the lines and give some basic results. Then, we prove Theorem~\ref{thm:main}.
\subsection{An action of $\agl(2,q)$ on the lines}

Let $q = p^k$ be a prime power, where $k\geq 1$. For $b\in \mathbb{F}_q^2$ and $A\in \gl(2,q)$, we let $(b,A): 
\mathbb{F}_q^2 \to \mathbb{F}_q^2$ be the affine transformation such that $(b,A)(v) := Av+b$. The \emph{affine group} 
$\agl(2,q)$ is the permutation group 
\begin{align*}
	\left\{(b,A) \mid A\in \gl(2,q),\ b\in \mathbb{F}_q^2  \right\},
\end{align*}
with the multiplication $(a,A)(b,B) := (a + Ab,AB)$.

Hence, $\agl(2,q)$ acts naturally on the vectors of $\mathbb{F}_q^2$. This action induces an action of $\agl(2,q)$ on the set $\Omega$ of all lines of $\mathbb{F}_q^2$ (i.e., the collection of all sets of the form $L_{u,v} := \left\{ u+tv \mid t \in \mathbb{F}_q \right\}$, where $u,v\in \mathbb{F}_q^2$ and $v\neq 0$). 
Recall that $\pg(1,\mathbb{F}_q) := \pg(1,q)$ is the set of all $1$-dimensional subspaces of the $\mathbb{F}_q$-vector space $\mathbb{F}_q^2$. The elements of $\pg(1,q)$ are exactly the lines containing $0\in \mathbb{F}^2_q$. By a simple counting argument, each vector of $\mathbb{F}_q^2 \setminus \{0\}$ determines a line, and each line passing through $0$ has $q-1$ points (excluding $0$). So there are $\frac{q^2-1}{q-1} = q+1$ subspaces in $\pg(1,q)$. 
For any line $\ell$ $\in \pg(1,q)$, we define $\Omega_\ell := \left\{ \ell + b \mid b\in \mathbb{F}_q^2 \right\}$. The set $\Omega_\ell$ consists of $\mathbb{F}_q^2$-shifts of the $1$-dimensional subspace $\ell$, thus its elements are affine lines of $\mathbb{F}_q^2$ that are parallel to $\ell$. Therefore, $\Omega :=\bigcup_{\ell \in \pg(1,q) } \Omega_\ell$ is exactly the set of lines of $\mathbb{F}_q^2$. Note that we can also view $\Omega$ as the lines of the incidence structure $\left(\mathbb{F}_q^2,L, \sim  \right)$, where $L = \{L_{u,v} \mid u,v\in \mathbb{F}_q^2,\ v\neq 0\}$ and $v\sim \ell$, for $v\in\mathbb{F}_q^2$ and $\ell \in L$, if and only if $v\in \ell$. This incidence structure is the affine plane $\ag(2,q)$.


As $\gl(2,q)$ acts transitively on $\pg(1,q)$, it is easy to see that $\agl(2,q)$ acts transitively on $\Omega$.
Since the elements of $\gl(2,q)\leq \agl(2,q)$ leave $\pg(1,q)$ invariant, for any $\ell\in \pg(1,q)$, the set $\Omega_\ell$ is invariant by the action of an element of $\agl(2,q)$ or is mapped to some other $\Omega_{\ell^\prime}$, where $\ell^\prime \in \pg(1,q) \setminus\{ \ell\}$. That is, $\Omega_\ell$ is a block for the action of $\agl(2,q)$ on $\Omega$. Therefore, $\agl(2,q)$ acts imprimitively on $\Omega$.

As elements of $\agl(2,q)$ are affine transformations, the pair of parallel lines $(l,l^\prime)\in \Omega_\ell \times \Omega_\ell$ can be mapped by $\agl(2,q)$ to any other pair of parallel lines. However, if $(l,l^\prime)\in \Omega_\ell \times \Omega_{\ell^\prime}$, for distinct $\ell,\ell^\prime \in \pg(1,q)$, then no element of $\agl(2,q)$ can map $(\ell,\ell^\prime)$ to a pair of parallel lines. In addition, one can prove that any pair of non-parallel lines can be mapped to any other pair of non-parallel lines. In other words, $\agl(2,q)$ acting on $
\Omega^2$ has exactly $3$ orbits. We formulate this result as the following lemma.
\begin{lem}
	The group $\agl(2,q)$ acting on $\Omega$ is a rank 3 imprimitive group.\label{eq:rank}
\end{lem}

\subsection{Action of Singer subgroups of $\gl(2,q)$ as subgroups of $\agl(2,q)$}
We recall that for $n\geq 1$, $\gl(n,q)$ admits elements of order $q^{n}-1$. These elements  are called \emph{Singer cycles}, and a subgroup of order $q^n-1$ generated by a Singer cycle is called a \emph{Singer subgroup.}
We recall the following observation about Singer cycles.

\begin{prop}
If $A$ is a Singer cycle of $\gl(2,q)$, then the subgroup $\langle A \rangle$ acts regularly on $\mathbb{F}_q^2 \setminus\{0\}$.\label{eq:regular}
\end{prop}

For any matrix $C\in \gl(2,q)$, we define 
\begin{align*}
G_q(C) := \left\{ (b,B) \mid B \in \langle C \rangle,\ b\in \mathbb{F}_q^2 \right\}.
\end{align*}

Now, let $A$ be an arbitrary Singer cycle of $\gl(2,q)$. 
By Proposition~\ref{eq:regular}, it is easy to see that the action of 
$$H_q := \left\{ (0,B) \in \agl(2,q) \mid B \in \langle A \rangle \right\}$$ 
on $\pg(1,q)$ is transitive. The latter implies that the action of the subgroup $G_q(A) $ on $\Omega$ is transitive. To see this, let $\ell = \ell_0 +b$ and $\ell^\prime = \ell_0^\prime +b^\prime$ be two lines in $\Omega$ such that $\ell_0$ and $\ell_0^\prime$ are $1$-dimensional subspaces  and $b,b^\prime \in \mathbb{F}_q^2$. By transitivity of $H_q$ on $\pg(1,q)$, there exists $(0,B) \in H_q$ such that $(0,B)(\ell_0) = \ell_0^\prime$. Hence, 
\begin{align*}
	( b^\prime -Bb,B) (\ell) &= ( b^\prime - Bb,B)(\ell_0 +b) 
	= B\ell_0 + Bb + b^\prime -Bb = \ell_0^\prime + b^\prime = \ell^\prime.
\end{align*}
Thus, $G_q(A)$ is transitive. It is straightforward to verify that for any $\ell\in \pg(1,q)$, $\Omega_\ell$ is a block of $G_q(A)$.
Therefore, we have the following.
\begin{prop}
	The group $G_q(A)$ \mbox{ acts imprimitively on } $\Omega$  and $\Omega_{\ell}$ is a block of $G_q(A)$, for any $\ell \in \pg(1,q)$.
\end{prop}

\subsection{Kernel of the action of $G_q(A)$}
In this subsection, we study the kernel of the action of $G_q(A)$ on the system of imprimitivity $\{\Omega_\ell \mid \ell \in \pg(1,q) \}$.

To avoid any confusion, we use the notation $\stab_{G_q(A)}(l)$ in the remainder of Section~\ref{sect:proof1} to denote the point-stabilizer of $\ell \in \Omega$ in $G_q(A)$, instead of the standard notation used in the theory of permutation groups. Similarly, for any $S\subset \Omega$, we use the notation $\stab(G_q(A),S)$ for the setwise stabilizer of $S$ in $G_q(A)$.

By Lemma~\ref{eq:rank}, the action of $\agl(2,q)$ on $\Omega$ has a unique system of imprimitivity, namely the set $\left\{ \Omega_\ell \mid \ell \in \pg(1,q) \right\}$. 
Define
$$M_q := \bigcap_{\ell \in \pg(1,q)} \stab({G_q}(A), \Omega_\ell).$$

We prove the following lemma.
\begin{lem}
	The affine transformation $(b,B) \in M_q$ if and only if there exists $k\in \mathbb{F}_q^*$ such that $B = k I$, where $I$ is the $2\times 2$ identity matrix.\label{lem:stab_blocks}
\end{lem}
\begin{proof}
	It is easy to see that if $B= kI$, for some $k\in \mathbb{F}_q^*$, then $(0,B)$ fixes every element of $\pg(1,q)$. Therefore, $(b,B)$ leaves $\Omega_\ell$ invariant for any $\ell \in \pg(1,q)$.
	
	If $(b,B) \in M_q$, then $(0,B)$ fixes every element of $\pg(1,q)$. In particular, there exists $k_1,k_2\in \mathbb{F}_q^*$ such that 
	\begin{align*}
		(0,B) 		\begin{bmatrix}
			1\\
			0
		\end{bmatrix}
		&= B 
		\begin{bmatrix}
			1\\
			0
		\end{bmatrix}
		 = k_1 
		 \begin{bmatrix}
		 	1\\
		 	0
		 \end{bmatrix},
		 \mbox{ and } \ \ \ 
		 (0,B)
		 \begin{bmatrix}
		 	0\\
		 	1
		 \end{bmatrix}
		 = B
		 \begin{bmatrix}
		 	0\\
		 	1
		 \end{bmatrix}
		 =k_2
		 \begin{bmatrix}
		 0\\
		 1
		 \end{bmatrix}.
	\end{align*}
	Therefore, the matrix $B = \operatorname{diag}(k_1,k_2)$. The $1$-dimensional subspace generated by the vector
	$u = \begin{bmatrix}
	1\\
	1
	\end{bmatrix}$
	forces $k_1 = k_2$, since $Bu = k u$ for some $k \in \mathbb{F}_q^*$. Hence $B = k I$. 
\end{proof}
We present an immediate corollary of this.
\begin{cor}
	The subgroup $M_q$ of $G_q(A)$ is intersecting.\label{cor:int}
\end{cor}
\begin{proof}
	It suffices to prove that any element of $M_q$ has a fixed point. Let  $(b,kI_2) \in M_q$. If $k=1$, then it is obvious that $(b,I)$ fixes every line in the block $\Omega_\ell$, where $\ell \in \pg(1,q)$ such that $b \in \ell$.
	
	If $k\neq 1$, then we prove that there exist $\beta \in \mathbb{F}_q^2$ such that for any $\ell \in \pg(1,q)$, $(b,kI)$ fixes the line $\ell+\beta$. If $(b,kI)$ fixes this line, then we must have
	\begin{align*}
		(b,kI)(\ell + \beta) &= k\ell + k \beta +b\\
		&= \ell + k \beta + b = \ell + \beta.
	\end{align*}
	In other words, we should find $\beta$ such that $(1-k)\beta - b \in \ell$, for any $\ell \in \pg(1,q)$. For $\beta =(1-k)^{-1} b$, we have $(1-k)\beta - b=0 \in \ell$. Moreover, the solution $\beta=(1-k)^{-1} b$ does not depend on $\ell$ since every element of $\pg(1,q)$ contains $0$.
	
	We conclude that when $k = 1$, then $(b,kI)$ fixes every line of the block $\Omega_{\ell}$, with $b\in \ell$ and if $k \neq 1$, then $(b,kI)$ fixes every line of the form $\ell + (1-k)^{-1} b \in \Omega$, for any $\ell \in \pg(1,q)$.
\end{proof}

We prove the following lemma about the relation between the kernel of the action of  $G_q(A)$ on $\{ \Omega_\ell \mid \ell \in \pg(1,q)\}$ and the subgroup $\Fix(G_q(A))$ generated by the non-derangements of $G_q(A)$.

\begin{lem}
	The subgroup $\Fix(G_q(A))$ is equal to $ M_q$.\label{lem:equal}\label{lem:stabisfix}
\end{lem}
\begin{proof}
	Let $(b,kI)\in M_q$. In the proof of Corollary~\ref{cor:int}, we showed that a transformation of $(b,kI)$ either fixes every element of $\Omega_\ell$, for some $\ell\in \pg(1,q)$, or it fixes exactly one line in each $\Omega_{\ell}$. Therefore, $M_q\leq \Fix(G_q(A))$.

	Next, we will prove that the point-stabilizer $\operatorname{Stab}_{G_q(A)}(\ell)$ of $\ell$ in $G_q(A)$ is a subgroup of $ M_q$, for any $\ell\in \Omega$. First, let $\ell \in \pg(1,q)$ be the line that contains the point 
	$\begin{bmatrix}
		1\\
		0
	\end{bmatrix} \in \mathbb{F}_q^2$. Observe that for $b\in \mathbb{F}_q^2$, $\ell + b = \ell$ if and only if $b \in \ell$. Therefore, the affine transformation $(b,kI)\in \operatorname{Stab}_{G_q(A)}(\ell)$, for any $b\in \ell$ and $k\in \mathbb{F}_q^*$. There are $q(q-1)$ affine transformations of this form in $\operatorname{Stab}_{G_q(A)}(\ell)$. Arguing by the size of the stabilizer of $\ell$ in $G_q(A)$, we have
	\begin{align*}
		|\operatorname{Stab}_{G_q(A)}(\ell)| &= \frac{q^2(q^2-1)}{q(q+1)} = q(q-1).
	\end{align*}
	We conclude that the point-stabilizer of $\ell$ in $G_q(A)$ is  $$\operatorname{Stab}_{G_q(A)}(\ell) = \left\{ (b,B)\in G_q \mid B = kI,\ b\in \ell,\ k\in \mathbb{F}_q^* \right\} \leq M_q.$$ 
	Since $M_q \triangleleft G_q(A)$ and $G_q(A)$ is transitive, we have $\operatorname{Stab}_{G_q(A)}(\ell) \leq M_q$ for any $\ell \in \Omega$. Therefore, $\Fix(G_q(A)) \leq M_q$. This completes the proof.
\end{proof}

\subsection{Proof of Theorem~\ref{thm:main}}
We prove that the derangement graph $\Gamma_{G_q(A)}$ of $G_q(A)$ is a complete $(q+1)$-partite graph.
By Corollary~\ref{cor:int}, $M_q$ is intersecting, and by Lemma~\ref{lem:equal}, we have $M_q = \Fix(G_q(A))$. Therefore, $\Gamma_{G_q(A)}$ is a complete $k$-partite graph, where $k = [G_q(A): M_q] = \frac{q^2(q^2-1)}{q^2(q-1)} = q+1$.

Note that Lemma~\ref{lem:stabisfix} is crucial to our proof. Indeed, the subgroup generated by the permutations with fixed points in $\agl(2,q)$ acting on $\Omega$, i.e., $\Fix(\agl(2,q))$, is the whole of $\agl(2,q)$; whereas the stabilizer of its unique system of imprimitivity is the proper subgroup $M_q$.

\section{Proof of Theorem~\ref{thm:main-construction}}\label{sect:proof2}
We will construct a transitive permutation group $G$ of degree $n=4\ell$  acting on $[n]$, where $\ell$ is an odd natural number. The derangement $\Gamma_{G}$ of this group $G$ will be a complete multipartite graph with $\frac{n}{2}$ parts. The group $G$ that we will construct is isomorphic to 
\begin{align*}
	(\underbrace{C_2 \times C_2\times \ldots \times C_2}_{\ell-1}) \rtimes D_{\ell},
\end{align*}
where $D_\ell$ is the dihedral group of order $2\ell$.

\subsection{Kernel of the action}
We would like to construct $G$ so that it will have a system of imprimitivity 
\begin{align*}
	\mathcal{B} = \left\{ \{i,i+1\} \mid  \mbox{ for } i \in [n]\cap \left(2\mathbb{Z}+1\right) \right\}.
\end{align*}
For any $i,j\in \left(4\mathbb{Z}+1\right)\cap[n]$, define $\sigma_i := (i\ \ i+1)(i+2\ \ i+3)$ and $\pi_j := \sigma_j \sigma_{4\ell-3}$. Let $S = \{ \pi_j \mid j\in \left(4\mathbb{Z}+1\right)\cap[n]\}$. Notice that $|S| = \ell$, however, $\pi_{4\ell-3} = id \in S$.
We consider the permutation group $H = \langle S \rangle$. It is easy to see that 
$$H \cong \underbrace{C_2 \times C_2\times \ldots \times C_2}_{\ell-1}.$$ 
Moreover, for any fixed $k \in [n]\cap (4\mathbb{Z} +1)$, any subset of the form $\left\{ \sigma_i \sigma_k \mid i \in [n]\cap (4\mathbb{Z}+1)\right\}$ generates $H$.

A permutation of $H$ either fixes, pointwise, an element of $\mathcal{B}$ or interchanges the pair of elements in a set of $\mathcal{B}$. Therefore, $H$ leaves $\mathcal{B}$ invariant. Any $g\in H$ can be written in the form 
\begin{align}
	g = \prod_{j\in [n] \cap (4\mathbb{Z}+1) }\pi_{j}^{k_j},\label{eq:expression}
\end{align}
for some $k_j\in \{0,1\}$. Since $\pi_{4\ell-3} = id$, there are at most $\ell-1$ (which is even) permutations of the form $\pi_j$ in the expression of $g$ in \eqref{eq:expression}. If the number of non-identity terms in \eqref{eq:expression} is even, then $g$ fixes the points $4\ell-3,\ 4\ell-2,\ 4\ell-1, \mbox{ and } 4\ell$. If the number of non-identity terms in \eqref{eq:expression} is odd, then there exists $j\in [n]\cap (4\mathbb{Z}+1)$, $j\neq 4\ell-3$, such that $k_j =0$ (because $\ell-1$ is even). Therefore, $g$ fixes the elements $j,\ j+1,\ j+2 \mbox{, and } j+4$. We conclude that
\begin{align}
	\mbox{ $H$ is an intersecting subgroup of degree $n=4\ell$.}\label{eq:H_int}
\end{align}

The group $G$ will be defined so that $H = \Fix(G)$.

\subsection{ Action of a dihedral group on $H$}
First, we give a permutation $c$, which is a product of four disjoint $\ell$-cycles. Then, we construct a transposition $\tau$ so that $\tau c \tau^{-1} = c^{-1}$. In other words, $\langle c,\tau \rangle = D_\ell$. This subgroup will act on $H$ so that $\langle H, c,\tau \rangle$ is transitive.

For any $i\in \mathbb{Z}$, define $A_i := (i\ \ i+4\ \ \ldots\ \ i+4k\ \ \ldots\ \ i+4(\ell-1))$ to be the permutation of order $\ell$, whose entries in the cycle notation are those of an arithmetic progression of step $4$, and with initial value $i$. Let $$c := A_1A_2A_3A_4.$$ 
We note that $A_1,A_2,A_3,\mbox{ and }A_4$ are pairwise disjoint $\ell$-cycles.
Consider the permutation 
\begin{align*}
	\tau := (1\ 3) (2\ 4)\prod_{i\in \{1\ 2\ \ldots\ \ell-1\}} \left(1+4i\ \ 3 +4(\ell-i) \right)\left( 2+4i\ \ 4 +4(\ell-i) \right).
\end{align*}
The transpositions in the expression of $\tau$ are also pairwise disjoint. Moreover, $\tau$ is a derangement of $\sym(n)$. The following conditions are satisfied by $\tau$
\begin{align}
	\begin{split}
	\tau A_1 \tau^{-1} &= A^{-1}_3,\\
	\tau A_2 \tau^{-1} &= A^{-1}_4,\\
	\tau A_3 \tau^{-1} &= A^{-1}_1,\\
	\tau A_4 \tau^{-1} &= A^{-1}_2.
	\end{split}\label{eq:conjugates}
\end{align}
From \eqref{eq:conjugates}, we deduce that $\tau c \tau^{-1} = c^{-1}$. We conclude that $\langle \tau,c \rangle \cong D_\ell$.

Next, we see how the subgroup $\langle c,\tau\rangle$ acts on $H$. For $i \in [n] \cap (4\mathbb{Z} +1)$ with $i\neq 4\ell-3$, we have 
\begin{align*}
	\nu_i := c \pi_i c^{-1} = c \sigma_i \sigma_{4\ell-3} c^{-1} = \sigma_{i+4}\sigma_1.
\end{align*}
Since $\{ \nu_i \mid i \in [n] \cap (4\mathbb{Z} + 1) \}$ also generates $H$, we conclude that $cHc^{-1} = H$. In addition, for any $i \in [n]\cap (4\mathbb{Z} +1)$, we have
\begin{align*}
	\mu_i := \tau \pi_i \tau^{-1} = \tau \sigma_i\sigma_{4\ell-3} \tau^{-1} =  \sigma_{\tau(i+2)}\sigma_{5}.
\end{align*}
Since $\{\mu_i \mid i\in [n]\cap (4\mathbb{Z} +1)\}$ also generates $H$, we have $\tau H \tau^{-1} = H$.

We conclude that $G := H\langle \tau,c \rangle$ is a permutation group of degree $4\ell$. In addition, it is easy to see that $H\cap \langle \tau,c \rangle = \{id\}$, so we have $G = H \rtimes \langle \tau, c\rangle$. Furthermore, $G$ is transitive because 
\begin{enumerate}[1)]
	\item the orbits of $H\langle c \rangle$ are $\{ 1+4i \mid i \in \{0,1,2,\ldots,\ell-1\} \}\cup \{ 2+4i \mid i\in \{0,1,2,\ldots,\ell-1\} \}$ and $\{3+4i \mid i \in \{ 0,1,2,\ldots,\ell-1 \}\} \cup \{4+4i \mid i \in \{0,1,2,\ldots,\ell-1\}\}$, and
	\item the orbits of $\langle \tau \rangle$ are the sets of the form $\{ 1+4i,3+4(\ell-i) \},\ \{2+4i,4+4(\ell-i)\}$  where $i\in \{0,1,\ldots,\ell-1\}$ $,\ \{2,4\}$, and $\{1,3\}$.
\end{enumerate}
\subsection{Derangement graph of $G$}
The derangement graph of $G$ is a complete multipartite graph with $2\ell$ parts. To prove this, we need to show that $H$ is intersecting and $\Fix(G) = H$. We only need to prove the latter since $H$ is an intersecting subgroup (see \eqref{eq:H_int}).

On one hand, as the elements of $S$ all have fixed points, it is easy to see that $\langle S \rangle =H \leq \Fix(G)$. On the other hand, the subgroup $K = \langle \pi_5,\pi_9,\ldots,\pi_{4i+1},\ldots,\pi_{4\ell-7} \rangle \leq H$ fixes $1$; that is, $K\leq G_1$. Since $|K| = 2^{\ell-2}$ and $|G_1| = \frac{|G|}{4\ell} = 2^{\ell-2}$, we conclude that $G_1=K \leq H$. As $G$ is transitive, the point-stabilizers of $G$ are conjugate. Moreover, since $H \triangleleft G$ (because $G = H \rtimes \langle \tau,c \rangle$) and $G_1\leq H$, we can conclude that $G_i \leq H$, for any $i\in [n]$. Therefore, $\Fix(G)\leq H$.

In conclusion, we know that $\Fix(G) = H$ is intersecting. This is equivalent to $\Gamma_{G}$ being a complete multipartite graph, with $[G:H] = 2\ell$ parts.

\section{Proof of Theorem~\ref{thm:main2}}\label{sect:proof3}

We will prove that every transitive group of degree $2p$, for any odd prime $p$, has intersection density at most $2$ (Theorem~\ref{thm:main2}) by showing that there is a clique of size $p$ in the derangement graph of $G$. In this case, we have $\rho(G) \leq \frac{|\Omega|}{p} = 2$. Therefore, $1\leq \rho(G) \leq 2$ for any transitive group $G$ of degree $2p$. It is proved in \cite[Lemma~5.3]{meagher180triangles} that for any odd $\ell$, there is a transitive group of degree $2\ell$, whose intersection density is $2$. Therefore, we will have $I(2p) = 2$, for any odd prime $p$.

As $p\mid |G|$, by Cauchy's theorem, there exists $\sigma \in G$ whose order is $p$. Therefore, $\sigma$ is either a $p$-cycle or the product of two disjoint $p$-cycles. If the latter holds, then $\sigma$ is a derangement of $G$ and $\langle\sigma \rangle$ is then a clique of size $p$ in $\Gamma_{G}$. So, we may suppose that $\sigma$ is a $p$-cycle.
\subsection{Imprimitive case}  
Since $G \leq \sym(\Omega)$ is imprimitive of degree $2p$, a non-trivial block of imprimitivity of $G$ has size $2$ or $p$. Assume that
\begin{align*}
\sigma = (x_1\ \ x_2\ \ x_3\ \ \ldots\ \ x_p).
\end{align*}
As $p$ is an odd prime and $\sigma \in G$, it is easy to see that $G$ cannot have a system of imprimitivity consisting of sets of size $2$.
We suppose that $G$ has a set of imprimitivity $\mathcal{Q}$ consisting of two subsets of size $p$ of $\Omega$. It is easy to see that $\sigma$ cannot interchange the two blocks of $\mathcal{Q}$ since the support of $\sigma$ only has $p$ elements. Thus, $\sigma$ is in the setwise stabilizer of $\mathcal{Q}$. Suppose that $\mathcal{Q} = \left\{ B, B^\prime \right\}$, where $B=\{x_1,x_2,\ldots,x_p\}$ and $B^\prime = \{y_1,y_2,\ldots,y_p\}$. As $G_{y_1}$ and $G_{x_1}$ are conjugate, there exists an element $\sigma^\prime \in G_{x_1}$, which is a $p$-cycle.
As $\sigma^\prime$ is a $p$-cycle, it must fix $B$ pointwise and act as a $p$-cycle on $B^\prime$.

We conclude that the permutation $\sigma\sigma^\prime\in G$ is a product of two disjoint $p$-cycles. The subgroup $\langle \sigma \sigma^\prime\rangle$ is a clique of size $p$ of $\Gamma_{G}$.

\subsection{Primitive case}
Suppose that $G\leq \sym(\Omega)$ is primitive of degree $2p$. We derive the result of Theorem~\ref{thm:main2} from the following lemma.

\begin{lem}[\cite{wielandt2014finite}]
	Suppose that $p$ is an odd prime. A primitive group of degree $2p$ is either $2$-transitive or every non-identity element of a Sylow $p$-subgroup of $G$ is a product of two disjoint $p$-cycles.\label{lem:primitive}
\end{lem}

By Lemma~\ref{lem:primitive}, we conclude that $G$ is $2$-transitive or $G$ contains a derangement of order $p$. Hence, either $G$ has the EKR property \cite{meagher2016erdHos} (in which case $\rho(G) = 1$) or $\rho(G) \leq 2$.

This completes the proof of Theorem~\ref{thm:main2}.

\section{Further work}
	We finish this paper by posing some open questions.
	In Section~\ref{sect:proof3}, we proved that for any odd prime $p$, a transitive group $G$ of degree $2p$ has intersection density $1 \leq \rho(G) \leq 2$. It follows from the classification of finite simple groups that the only simply primitive groups (i.e., primitive groups that are not $2$-transitive) of degree $2p$ are $\alt(5)$ and $\sym(5)$, both of degree $10$. Using \verb|Sagemath| \cite{sagemath}, the largest intersecting family of $\alt(5)$ is of size $12$, whereas its stabilizer of a point has size $6$. The largest intersecting family of $\sym(5)$ is $12$, which equals the size of its point-stabilizers. We conclude that the group $\alt(5)$ of degree $10$ has the largest intersection density among all primitive groups of degree $2p$, for every odd prime $p$. 
	
	For the imprimitive case, there are infinitely many examples of transitive groups with intersection density equal to $2$. In \cite[Lemma~5.3]{meagher180triangles}, the authors gave a family of transitive groups of degree $2\ell$, for any odd $\ell$, whose derangement graphs are $\ell$-partite and whose intersection density is $2$.  Based on a non-exhaustive search on the small transitive groups of degree $2p$ (where $p$ is an odd prime) available on \verb|Sagemath|, we are inclined to believe that the intersection density of a transitive group of degree $2p$, where $p$ is an odd prime, is an integer. We ask the following question.
\begin{qst}
	Does there exist an odd prime $p$ and  a transitive group $G$ of degree $2p$ such that $\rho(G)$ is not an integer?
\end{qst}

In Theorem~\ref{thm:main-construction}, we proved that there exists a family of transitive groups of degree $4\ell$, for any odd $\ell$, with complete $2\ell$-partite derangement graphs. This further confirms the validity of \cite[Conjecture~6.6 (1)]{meagher180triangles} (see also Conjecture~\ref{conj:complete-multipartite}) about the existence of transitive groups of any degree $n$ which is even but not a power of $2$, with a complete $\frac{n}{2}$-partite derangement graph.
\begin{prob}
	For any  odd $\ell$ and an integer $i \geq 3$, find a transitive group of degree $2^i\ell$ whose derangement graph is a complete $2^{i-1}\ell$-partite graph.
\end{prob}

In Section~\ref{sect:proof1}, we gave an example of a transitive group of degree $q(q+1)$, where $q$ is a prime power, whose intersection density is equal $q$. A non-exhaustive search on small transitive groups of degree $q(q+1)$, which are available on \verb|Sagemath|, shows that the largest intersection density for these groups is $q$. We ask the following question.
\begin{qst}
	Does there exist a transitive group $G$ of degree $q(q+1)$, where $q $ is a prime power, such that $\rho(G) > q$?\label{qst:main}
\end{qst}

Our motivation to work on the EKR property for the transitive group in Section~\ref{sect:proof1} comes from studying the EKR property for $\agl(2,q)$ acting on the lines of $\ag(2,q)$ (see Section~\ref{sect:proof1}), where $q$ is a prime power. Observe that if $H$ and $G$ are transitive permutation groups acting on $\Omega$ and $H\leq G$, then $\Gamma_H$ is an induced subgraph of $\Gamma_{G}$. Using the No-Homomorphism Lemma \cite{albertson1985homomorphisms}, one can prove that $\alpha (\Gamma_{G}) \leq \alpha(\Gamma_H) \frac{|G|}{|H|}$. We deduce from this inequality that if $H$ has the EKR property, then so does $G$. Moreover, $\rho(G) \leq \rho(H)$. 

Recall that the subgroup $G_q(A)$ defined in Section~\ref{sect:proof1} is a subgroup of $\agl(2,q)$ acting on the lines of $\ag(2,q)$. Using the result from the previous paragraph, we know that 
\begin{align*}
	\rho(\agl(2,q)) \leq \rho(G_q(A)) = \frac{q^2(q-1)}{q(q-1)}=  q,
\end{align*}
where $q$ is a prime power and $A$ is a Singer cycle of $\gl(2,q)$. However, we believe that this bound is not sharp. Indeed, from the observation of the behavior of the intersection density of $\agl(2,q)$ ($q \in \{3,4,5,7,8\}$) acting on the lines of $\ag(2,q)$, we make the following conjecture.
\begin{conj}
	For any $\varepsilon >0$, there exists a prime power $q_0$, such that for any prime power $q\geq q_0$, $0\leq \rho(\agl(2,q))-1 \leq \varepsilon$. In particular, $\rho(\agl(2,q)) \in \mathbb{Q} \setminus \mathbb{N}$, for any prime power $q$.
\end{conj}

\noindent\textsc{Acknowledgment}: The author would like to thank Roghayeh Maleki, Karen Meagher and Shaun Fallat for proofreading and helping improve the presentation of this paper. The author is also grateful to the two anonymous referees for their valuable comments and suggestions.

\bibliographystyle{plain}

\begin{thebibliography}{10}
	
	\bibitem{ahmadi2014new}
	B.~Ahmadi and K.~Meagher.
	\newblock A new proof for the {E}rd{\H{o}}s--{K}o--{R}ado theorem for the
	alternating group.
	\newblock {\em Discrete Mathematics}, {\bf 324}:28--40, 2014.
	
	\bibitem{ahmadi2015erdHos}
	B.~Ahmadi and K.~Meagher.
	\newblock The {E}rd{\H{o}}s--{K}o--{R}ado property for some 2-transitive groups.
	\newblock {\em Annals of Combinatorics}, {\bf 19}(4):621--640, 2015.
	
	\bibitem{meagher2015erdos}
	B.~Ahmadi and K.~Meagher.
	\newblock The {E}rd{\H{o}}s--{K}o--{R}ado property for some permutation groups.
	\newblock {\em Australasian Journal of Combinatorics}, 61(1):23--41, 2015.
	
	\bibitem{albertson1985homomorphisms}
	M.~O. Albertson and K.~L. Collins.
	\newblock Homomorphisms of 3-chromatic graphs.
	\newblock {\em Discrete mathematics}, {\bf 54}(2):127--132, 1985.
	
	\bibitem{behajaina20203}
	A.~Behajaina, R.~Maleki, A.~T. Rasoamanana, and A~S. Razafimahatratra.
	\newblock $3 $-setwise intersecting families of the symmetric group.
	\newblock {\em arXiv preprint arXiv:2010.00229}, 2020.
	
	\bibitem{cameron2003intersecting}
	P.~J. Cameron and C.~Y. Ku.
	\newblock Intersecting families of permutations.
	\newblock {\em European Journal of Combinatorics}, {\bf 24}(7):881--890, 2003.
	
	\bibitem{Frankl1977maximum}
	M.~Deza and P.~Frankl.
	\newblock On the maximum number of permutations with given maximal or minimal
	distance.
	\newblock {\em Journal of Combinatorial Theory, Series A}, {\bf
		22}(3):352--360, 1977.
	
	\bibitem{deza1983erdos}
	M.~Deza and P.~Frankl.
	\newblock {E}rd\H{o}s--{K}o--{R}ado theorem--22 years later.
	\newblock {\em SIAM Journal on Algebraic Discrete Methods}, 4(4):419--431,
	1983.
	
	\bibitem{ellis2011intersecting}
	D.~Ellis, E.~Friedgut, and H.~Pilpel.
	\newblock Intersecting families of permutations.
	\newblock {\em Journal of the American Mathematical Society}, {\bf
		24}(3):649--682, 2011.
	
	\bibitem{erdos1961intersection}
	P.~Erd\H{o}s, C.~Ko, and R.~Rado.
	\newblock Intersection theorems for systems of finite sets.
	\newblock {\em The Quarterly Journal of Mathematics}, {\bf 12}(1):313--320,
	1961.
	
	\bibitem{frankl1986erdos}
	P.~Frankl and R.~M. Wilson.
	\newblock The {E}rd{\H{o}}s--{K}o--{R}ado theorem for vector spaces.
	\newblock {\em Journal of Combinatorial Theory, Series A}, {\bf
		43}(2):228--236, 1986.
	
	\bibitem{godsil2009new}
	C.~Godsil and K.~Meagher.
	\newblock A new proof of the {E}rd{\H{o}}s--{K}o--{R}ado theorem for
	intersecting families of permutations.
	\newblock {\em European Journal of Combinatorics}, {\bf 30}(2):404--414, 2009.
	
	\bibitem{godsil2016erdos}
	C.~Godsil and K.~Meagher.
	\newblock {\em {E}rd\H{o}s--{K}o--{R}ado {T}heorems: {A}lgebraic {A}pproaches}.
	\newblock Cambridge University Press, 2016.
	
	\bibitem{jordan1872recherches}
	C.~Jordan.
	\newblock Recherches sur les substitutions.
	\newblock {\em Journal de Math{\'e}matiques Pures et Appliqu{\'e}es}, {\bf
		17}:351--367, 1872.
	
	\bibitem{larose2004stable}
	B.~Larose and C.~Malvenuto.
	\newblock {S}table sets of maximal size in {K}neser-type graphs.
	\newblock {\em European Journal of Combinatorics}, {\bf 25}(5):657--673, 2004.
	
	\bibitem{li2020erd}
	C.~H. Li, S.~J. Song, and V.~Pantangi.
	\newblock {E}rd{\H{o}}s--{K}o--{R}ado problems for permutation groups.
	\newblock {\em arXiv preprint arXiv:2006.10339}, 2020.
	
	\bibitem{meagher20202}
	K.~Meagher and A.~S. Razafimahatratra.
	\newblock 2-intersecting permutations.
	\newblock {\em arXiv preprint arXiv:2005.00139}, 2020.
	
	\bibitem{meagher180triangles}
	K.~Meagher, A.~S. Razafimahatratra, and P.~Spiga.
	\newblock On triangles in derangement graphs.
	\newblock {\em Journal of Combinatorial Theory, Series A}, {\bf 180}:105390,
	2021.
	
	\bibitem{meagher2011erdHos}
	K.~Meagher and P.~Spiga.
	\newblock An {E}rd{\H{o}}s--{K}o--{R}ado theorem for the derangement graph of
	{PGL}(2, q) acting on the projective line.
	\newblock {\em Journal of Combinatorial Theory, Series A}, {\bf
		118}(2):532--544, 2011.
	
	\bibitem{meagher2016erdHos}
	K.~Meagher, P.~Spiga, and P.~H. Tiep.
	\newblock An {E}rd{\H{o}}s--{K}o--{R}ado theorem for finite 2-transitive
	groups.
	\newblock {\em European Journal of Combinatorics}, {\bf 55}:100--118, 2016.

	
	\bibitem{serre2003theorem}
	J.~Serre.
	\newblock On a theorem of {J}ordan.
	\newblock {\em Bulletin of the American Mathematical Society}, {\bf
		40}(4):429--440, 2003.
	
	\bibitem{spiga2019erdHos}
	P.~Spiga.
	\newblock The {E}rd{\H{o}}s--{K}o--{R}ado theorem for the derangement graph of
	the projective general linear group acting on the projective space.
	\newblock {\em Journal of Combinatorial Theory, Series A}, {\bf 166}:59--90,
	2019.
	
	\bibitem{sagemath}
	{The Sage Developers}.
	\newblock {\em {S}ageMath, the {S}age {M}athematics {S}oftware {S}ystem
		({V}ersion 8.9)}, 2020.
	\newblock {\tt https://www.sagemath.org}.
	
	\bibitem{wielandt2014finite}
	H.~Wielandt.
	\newblock {\em Finite permutation groups}.
	\newblock Academic Press, 2014.
	
	\bibitem{wilson1984exact}
	R.~M. Wilson.
	\newblock The exact bound in the {E}rd{\H{o}}s--{K}o--{R}ado theorem.
	\newblock {\em Combinatorica}, {\bf 4}(2-3):247--257, (1984).
	
\end{thebibliography}

\end{document}